\documentclass[12pt]{article}
\usepackage{mathtools}
\usepackage{graphicx}
\usepackage{amsmath,amsfonts,latexsym,amscd,amssymb,theorem}
\usepackage{graphicx}
\usepackage[ansinew]{inputenc}
\usepackage{epsfig}
\usepackage{amsmath}
\usepackage{amssymb}
\usepackage{afterpage}
\usepackage{selinput}
\usepackage{fancyhdr}
\newcommand{\ba}{\begin{eqnarray}}
\newcommand{\ea}{\end{eqnarray}}

\newtheorem{thm}{Theorem}[section]

\newtheorem{conjecture}{Conjecture}

\newtheorem{theorem}[thm]{Theorem}

\newtheorem{claim}[thm]{Claim}

\makeatletter
\newcommand*{\rom}[1]{\expandafter\@slowromancap\romannumeral #1@}
\makeatother

\date{}

\begin{document} 
\author{Batoul Tarhini$^{1,2}$} \footnotetext[1]{KALMA Laboratory, Faculty of Sciences, Lebanese University, Beirut, Lebanon.}
     \footnotetext[2]{LIB Laboratory, University of Burgundy, Dijon, France.}
\title{About the existence of oriented paths with three blocks}
\maketitle
\bibliographystyle{alpha}
\begin{center}\textbf{Abstract}
\end{center}
A path $P(k_{1},k_{2},k_{3})$ is an oriented path consisting of $k_{1}$ forward arcs, followed by $k_{2}$ backward arcs, and then by $k_{3}$ forward arcs. We prove the existence of any oriented path of length $n-1$ with three blocks having the middle block of length one in any $(2n-3)$-chromatic digraph, which is an improvement of the latest bound reached in this case.\\ Concerning the general case of paths with three blocks, we prove, after partitioning the problem into three cases according to the values of $k_{1}$,$k_{2}$ and $k_{3}$,  that the chromatic number of digraphs containing no $P(k_{1},k_{2},k_{3})$ of length $n-1$ is bounded above by $2(n-1)+k_{3}$, $2(n-1)+k_{2}+k_{3}-k_{1}$ and $2(n+k_{2}-1)-k_{1}$ in the three cases respectively.

\section{\textbf{Introduction}}
\pagenumbering{arabic}
In this paper, all digraphs considered are simple; there are no loops, no multiple edges and no circuits of length $2$. 
A digraph $D$ is defined by its vertex set $V(D)$ and its arc set $E(D)$. We denote by $v(D)$ the cardinality of $V(D)$. $D^{c}$ is defined to be the digraph obtained from $D$ by reversing the orientation of all the arcs in $D$. By ignoring the orientation of all arcs of a digraph $D$, we obtain the graph $G[D]$. Let $K\subseteq V(D)$; we denote by $D[K]$ the induced restriction of $D$ to $K$. We say that a digraph $D$ contains $H$ if $H$ is isomorphic to a subdigraph of $D$. If $D$ does not contain $H$, we say that $D$ is $H-$free. Let $H$ be a subdigraph of a digraph $D$, for every $v\in V(D)$, we define $N^{+}_{H}(v)=\{u\in H$ such that $(v,u)\in E(D)\}$ and $N^{-}_{H}(v)=\{u\in H$ such that $(u,v)\in E(D)\}$. For simplicity, $N^{+}_{D}(v)$ and $N^{-}_{D}(v)$ are written $N^{+}(v)$ and $N^{-}(v)$ respectively. \\
An \textit{oriented path} is obtained from a path by giving an orientation of each of its edges. If $P=x_{1}x_{2}\dots x_{n}$ is an oriented path, then the subpath $x_{i}x_{i+1}\dots x_{j}$ is denoted by $P_{[x_{i},x_{j}]}$ for every $1 \leq i < j \leq n$. A \textit{block} of an oriented path $P$ is defined to be a  maximal directed subpath of $P$. We denote by  $P(k_{1}, k_{2},\dots , k_{m})$ the oriented path formed of $n$ blocks of lengths $k_{1}, k_{2}, ..., k_{m-1}$ and $k_{m}$ respectively.\\ An \textit{oriented tree} is an orientation of a tree.
 An \textit {out-tree} (respectively, \textit{in-tree}) is an oriented tree in which all the vertices has indegree (respectively, outdegree) at most 1. This implies that in  an out-tree (respectively, in-tree) there exists exactly one vertex of indegree (respectively, outdegree) zero, and it is called the \textit{root}.\\ An \textit{outforest} (respectively, \textit{inforest}) is an oriented forest in which its connected components are out-trees (respectively, in-trees). The unique directed path in $F$ joining $u$ with the root is denoted by $P_{F}(u)$. If $P_{F}(u_{k})=u_{1}\dots u_{k-1} u_{k}$ where $F$ is a spanning outforest, then the vertex $u_{k-1}$ is denoted by ${N_{F}}^{-}(u_{k})$. We denote by $l_{F}(u)$, the level of a vertex $u \in F$ that is the order of the path $P_{F}(u)$. $L_{i}(F)$ is defined to be the set containing all the vertices $u\in V(F)$ such that $ l_{F}(u)=i$. \\We say that an arc $(u,v)\in E(D)$ is a forward arc with respect to an outforest (respectively, inforest) $F$ of $D$ if $l_{F} (u) < l_{F} (v)$ (respectively, $l_{F} (u) > l_{F} (v)$), otherwise we call it a backward arc.
  A \textit{spanning outforest} (respectively, \textit{spanning inforest}) $F$ of a digraph $D$ is said to be \textit{final outforest} (respectively, \textit{final inforest}) of $D$ if and only if for every $(u,v)$ backward arc with respect to $F$, $F$ contains a $vu$ (respectively, $uv$)-directed path. \\ It is clearly noticed that in a final outforest (respectively, inforest) $F$, $L_{i}(F)$ is stable for every $i$.\\ \\
Digraphs contained in every $n$-chromatic digraph are called $n$-universal.\\
 In 2015, El Sahili \cite{ce} proposed the following conjecture:\\ \begin{conjecture} \label{111} For $n\geq 8$, every oriented path of order $n$ is $n$-universal. \end{conjecture}
 In 1980, Burr \cite{jhp} conjectured the following:\\ \begin{conjecture} \label{22} Every oriented tree of order $n$ is $(2n-2)$-universal. \end{conjecture} This is a generalization of Sumner's conjecture which states that every oriented tree of order $n$ is contained in every tournament of order $2n-2$.\\
 Our interests are directed towards studying the universality of oriented paths which are, clearly, special types of oriented trees.\\ Before showing our result, we will state main previous results in this field.\\ Regarding oriented paths in general, Burr \cite{jhp} showed that any oriented path of length $n-1$ is $(n-1)^2$-universal.\\  In the case of tournaments, Havet and Thomasse \cite{fdo} showed that except for three particular cases, any tournament of order $n$ contains any oriented path of length $n-1$.
 Gallai \cite {khj} and Roy \cite{dar} proved that any directed path of length $n-1$ is $n$-universal.\\
El-Sahili \cite{kg} conjectured that every path of order $n \geq 4$ with two blocks is $n$-universal.\\
L. Addario-Berry et al. \cite{po} used strongly connected digraphs and a theorem of Bondy \cite{hjk} to prove the correctness of El Sahili's conjecture.
\\The case of three blocks is then our concern. M. Mourtada et al. proved in  \cite{sss} that the path $P(n-3,1,1)$ is $(n+1)$-universal. Tarhini et al. proved the existence of any $P(k,1,l)$ of length $n-1$ in any $(2n+2)$-chromatic digraph and that if $D$ is an $n$-chromatic digraph containing a Hamiltonian directed path, then it contains any $P(k,1,l)$ of length $n-1$.\\ In this paper, we improve the bound reached in  \cite{mb}, so that we prove the following:\\ \\
Let $D$ be a $(2n-3)$-chromatic digraph; then $D$ contains any path $P(k,1,l)$ of length $n-1$, with $n\geq 7$.\\
This confirms that the path $P(k,1,l)$ satisfies Conjecture \ref{22}. \\ \\
 Concerning the general case of three blocks paths $P(k,l,r)$, we partition the problem into three cases according to the values of $k,$ $l$ and $r$. We establish a linear bound for the chromatic number
 which is $2(n-1)+r$, $2(n-1)+l+r-k$ and $2(n+l-1)-k$ in the three cases of the problem respectively.
The reached bound is at most $3n-6$ for two cases of the problem and $4n-13$ for the remaining case in its worst scenario, and so it is the first linear bound that hits $3n$ and it is an improvement of the bound reached by El Joubbeh \cite{Mjb}. \\ \\ 
This paper is organized as follows:\\ $\bullet$ In Section \ref{2}, we deal with the path $P(k,1,l)$. \\ $\bullet$ In Section \ref{3}, we deal with the general case of paths with three blocks.\\ \\
Remark that if we prove that an $f(n)$-chromatic digraph $D$ contains any path $P(k_{1},k_{2},k_{3})$ for $k_{1}\leq k_{3}$, then it is trivial that $D^{c}$ contains the path $P(k_{1},k_{2},k_{3})$ because it has the same chromatic number as $D$. Moreover, the existence of a path $P(k_{1},k_{2},k_{3})$, for $k_{1}\leq k_{3}$, in the digraph $D^{c}$ results in the existence of the path  $P(k_{3},k_{2},k_{1})$ in $D$. So to prove the existence of any path  $P(k_{1},k_{2},k_{3})$ in a digraph $D$, it is sufficient to prove it for $k_{1}\leq k_{3}$ and the case when $k_{1}> k_{3}$ is deduced by using the digraph $D^{c}$. Therefore, in what follows it is sufficient to deal only with the case when $k_{1}\leq k_{3}$.

\section{The existence of the path $P(k,1,l)$} \label{2}
\subsection{An overview of the proof}
In this section we give the reader the intuition how the proof works, what are the most important steps of the proof and what is the new technique used in this setting. In the cases under study, the followed technique has served us in improving the latest previous bounds of the chromatic number. We also believe that it may be useful in dealing with different types of oriented paths. This technique is followed in the two theorems that we prove, but with some differences that have advantages when the second block is of length $1$ (Theorem \ref{m}). 
First, we start by defining a sequence $\{D_{i},i\in I\}$ of subdigraphs of the digraph $D$, by taking out at each step a directed path $P_{i}$ with all its neighbors $N_{i}$ where the length of the path is precised according to the case of the path we are studying its existence. Let $P$ and $N$ be the union of the paths $P_{i}$ and their neighbors $N_{i}$ respectively. The next step is partitioning the set of neighbors $N$ into subsets according to the localization of their neighbors on the the paths $P_{i}$. We suppose that the digraph is $P(k,l,r)$-free and start assigning the colors to its vertices. In the coloring process, we have benefit that there are no edges between one path and another, and that we know the chromatic number bound of the remaining subdigraph since it contains no directed path of the specified length. When dealing with the subsets forming the set of neighbors, we use the concept of final outforest and inforest, and proceed the coloring of each subset according to its definition based on the localization of its neighbors. At the end, we reach a contradiction with the given chromatic number and this completes the proof.

\subsection{The proof}
\textbf{\begin{theorem}\label{m} Let $D$ be a $(2n-3)$-chromatic digraph; then $D$ contains any path $P(k,1,l)$ of length $n-1$, with $n\geq 7$. \end{theorem}}
\begin{proof}
First we note that the case when $k=1$ or $l=1$ is treated in \cite{sss}, in which they proved that any $(n+1)$-chromatic digraph contains the path $P(n-3,1,1)$.
Thus, we assume in what follows that $k\neq 1$ and $l\neq 1$.
Suppose that $D$ is $P(k,1,l)$-free and, without loss of generality, suppose that $l\geq k$, and so $l\geq 3$.
Let: \begin{itemize}
\item $D_{1}=D$ 
 \item $D_{i+1}=D_{i}-(P_{i} \cup N(P_{i}))$ for every $1\leq i\leq r-1$, where :\begin{itemize} \item $P_{i}$ is a directed path of length $k+l-2$ in $D_{i}$. \item$N(P_{i})=\{v \in D_{i}-P_{i};$ $v$ has a neighbor in $P_{i}$\}.\item $r$ is the minimal value satisfying that $D_{r}$ contains no directed path of length $k+l-2$.\end{itemize} \end{itemize}
Let $D'=D[\displaystyle \bigcup _{i=1}^{r-1}V(P_{i})\cup D_{r}]$  and  $N=\displaystyle \bigcup _{i=1}^{r-1} N(P_{i})$.\\ By definition, we have: \begin{itemize} \item $v(P_{i})=k+l-1$, then $\chi(D[V(P_{i})])\leq k+l-1=n-3$ for every $1\leq i\leq r-1$.\item $D_{r}$ contains no directed path of length $k+l-2=n-4$, then $\chi(D_{r})\leq n-4$. \item$ P_{i+1}\in D_{i+1}=D_{i}-(P_{i}\cup N(P_{i}))$, then $uv\notin E(G[D])$ for every $u\in P_{i}$ and $v\in P_{j}$ with $i\neq j$ for all $1\leq i,j\leq r-1$, similarly for every $u\in D_{r}$ and $v\in P_{i}$.\end{itemize} So, $\chi(D')\leq n-3$. We color $D'$ by the set $\{1,2,...,k+l-1\}$.\\
Let $P_{i}=v^{i}_{1}\dots v^{i}_{k+l-1}$, and let $A $ and $B$ be two subsets of $N$ such that:\\$B=\displaystyle{\bigcup_{i=1}^{r-1}}  N^{+}(v_{k}^{i})\cup N(P_{i_{[v_{k+1}^{i},v^{i}_{k+l-1}]}})$ and $A=N-B$, \textit{i.e.} $A\subseteq \displaystyle{\bigcup_{i=1}^{r-1}}N(P_{i_{[v^{i}_{1},v_{k-1}^{i}]}})\cup N^{-}(v_{k}^{i}).$ \begin{claim} \label{A}: $\chi(D[A])\leq k+2$.\end{claim} \begin{proof}
Let $F_{A}$ be a final spanning outforest of $D[A]$.\\ Let $A'=D[\displaystyle\bigcup_{i\geq k+1}L_{i}(F_{A})]$ and $F'_{A}$ be the sub-outforest of $F_{A}$ spanning the vertices of $A'$. For every $v\in A',$ $l(P_{F_{A}}(v))\geq k$, then for all $1\leq i\leq r-1$, $v$ is not an outneighbor of a vertex in $P_{i_{[v^{i}_{1},v_{k-1}^{i}]}}$, say $v_{j}$, otherwise $P_{F_{A}}(v)\cup (v_{j},v)\cup P_{i_{{[v_{j},v^{i}_{k+l-1}]}}}$ contains a $P(k,1,l)$, a contradiction. So, for every $v\in A' , v$ is an inneighbor of a vertex in $P_{i_{[v^{i}_{1},v_{k}^{i}]}}$ for some $1\leq i\leq r-1$.\\ $A'$ is a bipartite digraph, otherwise there exists $e\in E(A')\backslash E(F'_{A})$, so we study two cases:\\ Case 1: $e=(x,y)$ is backward with respect to $F_{A}$.\\We have $x\in A'$, so $x$ an inneighbor of a vertex in $ P_{i_{[v^{i}_{1},v_{k}^{i}]}}$, say $v_{j}$. Thus, $P_{F_{A}}(y) \cup (x,y)\cup (x,v_{j}) \cup P_{i_{[v_{j},v_{k+l-1}^{i}]}}$ contains a $P(k,1,l)$, a contradiction.\\ Case 2 : $e=(x,y)$ is forward with respect to $F_{A}$.\\ Let $z=N^{-}_{F_{A}}(y)$, we have $z\in A'$ and so it is an inneighbor of a vertex in $ P_{i_{[v^{i}_{1},v_{k}^{i}]}}$, say $v_{j}$. Thus, $P_{F_{A}}(x)\cup (x,y)\cup (z,y)\cup (z,v_{j})\cup P_{i_{[v_{j},v^{i}_{k+l-1}]}}$ contains a $P(k,1,l)$, a contradiction.\\ Therefore, $ E(A')=E(F'_{A})$. Thus, $A'$ is an outforest and consequently, a bipartite digraph which we partition into two stable sets $A_{1}$ and $A_{2}$. So $\chi(D[A])\leq k+2$. \end{proof}$\blacksquare$
\begin{claim} $\chi(D[B])\leq l+2$. \end{claim} \begin{proof} Let $F_{B} $ be a final spanning inforest of $D[B]$.\\ $\bullet$ Let $B'=D[\displaystyle\bigcup_{i\geq l+1}L_{i}(F_{B})]$ and $F'_{B}$ be the sub-inforest of $F_{B}$ spanning the vertices of $B'$. In a similar way to the proof of claim \ref{A}, we prove that $E(B')=E(F'_{B})$. Consequently, $B'$ is an inforest, which we partition into two stable sets $B_{1}$ and $B_{2}$.\end{proof}$\blacksquare$ \begin{claim} \label{c} $\chi(D[A_{1}\cup B'])\leq 2$.\end{claim} \label{88} \begin{proof}
 We have $ \chi(B')\leq 2$. 
  Let $v\in A_{1}, (u,v) \notin E(D)$ for every $u\in B'$, since otherwise $P_{F_{A}}(v)\cup (u,v)\cup P_{F_{B}}(u)$ contains a $P(k,1,l)$, a contradiction.\\ Now we are going to prove that $v$ has at most one outneighbor in $B'$. Suppose that $v$ has two outneighbors $u_{1}$ and $u_{2}$ in $B'$. Without loss of generality, suppose that $l_{F_{B}}(u_{1})\leq l_{F_{B}}(u_{2})$. We have $ u_{2}\in B'$, then $u_{2}$ is not an inneighbor of a vertex in $ P_{i_{[v_{k+1}^{i},v^{i}_{k+l-1}]}}$, say $v_{j}$, since otherwise $P_{i_{[v^{i}_{1},v_{j}]}}\cup (u_{2},v_{j})\cup P_{F_{B}}(u_{2})$ contains a $P(k,1,l)$, a contradiction. Thus, $u_{2}$ is an outneighbor of a vertex in $P_{{i}_{[v_{k}^{i},v^{i}_{k+l-1}]}}$, say $v_{j}$. But $P_{i_{[v^{i}_{1},v_{j}]}}\cup (v_{j},u_{2}) \cup (v,u_{2})\cup (v,u_{1})\cup P_{F_{B}}(u_{1})$ contains a $P(k,1,l)$, a contradiction.\\ So $v$ has at most one neighbor in $B'$ and consequently, $\chi(D[A_{1}\cup B'])\leq 2$.\end{proof} $\blacksquare$ \begin{claim} Let $F_{r}$ be a final spanning outforest of $D_{r}$. If $ L_{k+l-2}(D_{r})\neq \emptyset$, then $\chi(D[B'\cup A_{1}\cup  L_{k+l-2}(D_{r})])\leq 2$  \end{claim} \label{55}\begin{proof} 
Let $v\in L_{k+l-2}(D_{r})$. We have $l\geq 3 \Rightarrow l_{F_{D_{r}}}(v)\geq k+1$, then:\\
 
$\bullet v $ has no inneighbor in $B'$ since otherwise, $P_{F_{D_{r}}}(v)\cup (u,v) \cup P_{F_{B}}(u)$ contains a $P(k,1,l)$ where $u$ is an inneighbor of $v$ in $B'$, a contradiction.\\

 $\bullet v$ has no inneighbor in $A'$, say $u$ since otherwise, $P_{F_{D_{r}}}(v)\cup(u,v)\cup (u,v_{j})\cup P_{i_{[v_{j},v_{k+l-1}^{i}]}}$ contains a $P(k,1,l)$ where $v_{j}$ is an outneighbor of $u$ in $P_{i_{[v_{1}^{i},v_{k}^{i}]}}$, a contradiction.\\ Now we are going to prove that $v$ has at most one outneighbor in $B'\cup A_{1}$. Suppose that $v$ has two outneighbors $x$ and $y$ in $B'\cup A_{1}$: \\ If $x$ and $y$ belong to $B'$, then suppose, without loss of generality, that $l_{F_{B}}(x)\leq l_{F_{B}}(y)$.\\ We have $ y\in B'$, then $y$ is an outneighbor of a vertex in $P_{{i}_{[v_{k}^{i},v_{k+l-1}^{i}]}}$, say $v_{j}$. Thus, $P_{i_{[v_{1}^{i},v_{j}]}}\cup (v_{j},y) \cup (v,y)\cup (v,x)\cup P_{F_{B}}(x)$ contains a $P(k,1,l)$, a contradiction.\\If $x$ and $y$ belong to $A_{1}$, then suppose, without loss of generality, that $l_{F_{A}}(x)\leq l_{F_{A}}(y)$. We have $y \in A_{1}$, then $y$ is an inneighbor of a vertex in $P_{i_{[v_{1}^{i},v_{k}^{i}]}}$, say $v_{j}$. Thus, $P_{F_{A}}(x)\cup (v,x)\cup (v,y)\cup(y,v_{j})\cup P_{i_{[v_{j},v_{k+l-1}^{i}]}}$ contains a $P(k,1,l)$, a contradiction. \\ If $x\in A_{1}$ and $y\in B'$, then we have $P_{F_{A}}(x)\cup (v,x) \cup (v,y)\cup P_{F_{B}}(y)$ contains a $P(k,1,l)$, a contradiction.\\ Therefore $v$ has at most one neighbor in $B'\cup A_{1}$. By claim 2.4, we have  $\chi(B'\cup A_{1})\leq 2$, so $\chi(D[B'\cup A_{1}\cup  L_{k+l-2}(D_{r})])\leq 2$. $\blacksquare$ \end{proof}\\ \\
 Now we are ready to continue the proof of Theorem \ref{m}:\\ By definition, $A$ has neighbors only in $\displaystyle{\bigcup _{i=1}^{r-1}} P_{i_{[v^{i}_{1},v_{k}^{i}]}}$. So, we choose the coloring of $D'$ such that the colors of $V(\displaystyle{\bigcup _{i=1}^{r-1}} P_{i_{[v^{i}_{1},v_{k}^{i}]}})$ are picked from the set $\{1,2,\dots ,k\}$, then the colors $\{k+1,\dots k+l-1\}$ are not used in the coloring of $N_{\{P_{1}\cup \dots P_{r-1}\}}(A)$.\\ 
 $D_{r}$ contains no path of length $ k+l-2$, so $\chi(D_{r})\leq k+l-2$. Recall that $F_{r}$ is a final spanning outforest of $D_{r}$, we color $L_{i}(F_{D_{r}})$ by the color $i$ for every $i$.\\ If $\chi(D_{r})\leq k+l-3$, then we color $D_{r}$ by the set $\{1,\dots,k+l-3\}$.\\ Otherwise, $ L_{k+l-2}(D_{r})\neq \emptyset$. We have by claim $2.5$, $\chi(D[B'\cup A_{1}\cup  L_{k+l-2}(D_{r})])\leq 2$, so we can recolor $L_{k+l-2}(D_{r})$ by the set of colors given to $B'$.\\ Thus, in both cases, the colors $\{k+l-2,k+l-1\}$ are not used in the coloring of $D_{r}$.\\ Now we conclude that the colors $\{k+l-2,k+l-1\}$ are not used in the coloring of $N_{D'}(A)$, so we color $L_{1}(F_{A})$ and $A_{2}$ by these two colors.\\
Finally, $\chi(D)\leq \chi( L_{1}(F_{A})\cup A_{2}\cup D'-(  L_{k+l-2}(D_{r})))+\chi(B\cup A_{1}\cup  L_{k+l-2}(D_{r}))+\chi(A-(A_{1}\cup A_{2}\cup L_{1}(F_{A})))$\\ $\leq (n-3)+(l+2)+(k-1)=n-3+n-1=2n-4$, a contradiction.$\blacksquare$ \end{proof}

\section{The existence of the path $P(k,l,r)$ in digraphs with chromatic number bounded from above} \label{3}
It is worthy to note that the case when $k=1$ or $r=1$ is treated in \cite{dm}, in which D. Al Mniny proved that any $(2n-2)$-chromatic digraph contains the path $P(k,l,1)$ of length $n-1$. Moreover, the case when $l=1$ is treated in Theorem \ref{m} in this paper.
Thus, in what follows, we suppose that $1\notin \{k,l,r\}$.

\begin{theorem} \label{new3}
Let $D$ be a $P(k,l,r)$-free digraph with $k+l+r=n-1$ and $k\leq r$ , then \begin{equation}
\chi(D)\leq \left\{
\begin{array}{rcl}
 2(n-1)+r \mbox{     if }  l\leq k \leq r \\ \\
 2(n-1)+l+r-k\mbox{ if } k\leq l \leq r\\ \\
 2(n+l-1)-k \mbox{ if } k\leq r \leq l
 
\end{array}\right.
\end{equation}

\end{theorem} 
\begin{proof}
Let: \begin{itemize}
\item $D_{1}=D$ 
 \item $D_{i+1}=D_{i}-(P_{i} \cup N(P_{i}))$ for every $1\leq i\leq h-1$, where :\begin{itemize} \item $P_{i}$ is a directed path of length $k+r-2$ in $D_{i}$. \item$N(P_{i})=\{v \in D_{i}-P_{i};$ $v$ has a neighbor in $P_{i}$\}.\item $h$ is the minimal value satisfying the fact that $D_{h}$ contains no directed path of length $k+r-2$.\end{itemize} \end{itemize}
 Let $D'=D[ \displaystyle \bigcup _{i=1}^{h-1} V(P_{i})\cup D_{h}]$  and  $N=\displaystyle \bigcup _{i=1}^{h-1} N(P_{i})$.\\ By definition, we have: \begin{itemize} \item $v(P_{i})=k+r-1$, then $\chi(D[V(P_{i})])\leq k+r-1$ for every $1\leq i\leq h-1$.\item $D_{h}$ contains no directed path of length $k+r-2$, then $\chi(D_{h})\leq k+r-2$. \item$ P_{i+1}\in D_{i+1}=D_{i}-(P_{i}\cup N(P_{i}))$, then $uv\notin E(G[D])$ for every $u\in P_{i}$ and $v\in P_{j}$ with $i\neq j$ for all $1\leq i,j\leq h-1$, similarly for every $u\in D_{h}$ and $v\in P_{i}$.\end{itemize} So, $\chi(D')\leq k+r-1$. Color $D'$ by $\{1,2,...,k+r-1\}$.\\
Let $P_{i}=v^{i}_{1}\dots v^{i}_{k+r-1}$, and let $A,B,C, and$ $H $ be four subsets covering $N$ such that:\\
$H=\displaystyle \bigcup_{i} N^{+}({P_{i}}_{[v_{k}^{i} ,v_{k+r-1}^{i}]})$\\ $C=\displaystyle \bigcup_{i} N^{-}({P_{i}}_{[v_{k+1}^{i} ,v_{k+r-1}^{i}]})\cap (N-H)$\\$B=\displaystyle \bigcup_{i} N^{+}({P_{i}}_{[v_{1}^{i} ,v_{k-1}^{i}]})\cap (N-(H\cup C))$\\ $A=N-(H\cup C\cup B)$, $i.e.$ $A\subseteq \displaystyle \bigcup_{i} N^{-}({P_{i}}_{[v_{1}^{i} ,v_{k}^{i}]}) $.\\
Now let $F_{A}$ and $ F_{B}$ be final spanning outforests of $A$ and $B$ respectively. Let $F_{C}$ and $ F_{H}$ be final spanning inforests of $C$ and $H$ respectively. Suppose, without loss of generality, that $k\leq r$. Then it is enough to study the following three cases:\\ Case 1 : $l\leq k \leq r$.\\
Let $A'=\displaystyle \bigcup_{i\geq k+1}L_{i}(F_{A})$ and $B'=\displaystyle \bigcup_{i\geq l} L_{i}(F_{B})$.\\ We have $uv \notin E(G[D])$ for every  $u\in A'$ and $v\in B'$ since otherwise, suppose without loss of generality, that $(u,v)\in E(D)$, then $P_{F_{A}}(u)\cup (u,v) \cup P_{F_{B}}(v)_{[z,v]}\cup (v_{j},z) \cup {P_{i}}_{[v_{j},{v^{i}}_{k+r-1}]}$ contains a $P(k,l,r)$ where $P_{F_{B}}(v)_{[z,v]}$ is the subpath of $P_{v}(F_{B})$ of length $l-1$ and $v_{j}$ is an inneighbor of $z$ on ${P_{i}}_{[v_{1}^{i},v_{k-1}^{i}]}$ for some $i$.\\ \begin{claim}$\chi(A')\leq l+1$.\end{claim} \begin{proof} We have $A_{i}=\{\displaystyle \bigcup _{j\geq 0} L_{i+j(l+1)}\}$ is stable for every $i\geq k+1$, since otherwise let $uv\in E(G[D])$ such that $L_{F_{A}}(u)=i_{0}$ and $L_{F_{A}}(v)=i_{0}+j_{0}(l+1)$ for some $i_{0},j_{0}\in \mathbb{N^{*}}$ with $i_{0}\geq k+1$. Suppose, without loss of generality, that $(u,v)\in E(D)$. We have $P_{F_{A}}(u)\cup (u,v)\cup {P_{F_{A}}(v)}_{[z,v]}\cup (z,v_{j}) \cup {P_{i}}_{[v_{j},{v^{i}}_{k+r-1}]}$ contains a $P(k,l,r)$ where $P_{F_{A}}(v)_{[z,v]}$ is the subpath of $P_{F_{A}}(v)$ of length $l$ and $v_{j}$ is an outneighbor of $z$ in ${P_{i}}_{[v_{1}^{i},v_{k}^{i}]}$ for some $i$.\end{proof}\begin{claim} $\chi(B')\leq k+1$. \end{claim} \begin{proof} We define $B_{i}=\{\displaystyle \bigcup _{j\geq 0} L_{i+j(k+1)}\}$ for every $i\geq l$. Similarly to the proof of claim 3.2, we prove that $\chi(B')\leq k+1$. \end{proof}$\blacksquare$\\ \\We conclude, from claim 3.2 and claim 3.3, that $\chi(A'\cup B')\leq k+1$.\\ Therefore, $\chi(A\cup B)\leq \chi(A'\cup B')+\chi(A-A')+\chi(B-B')\leq k+1+k+l-1=k+k+l$.\\ \begin{claim} Let $F_{h}$ be a final spanning outforest of $D_{h}$. Then, $uv\notin E(G[D])$ for all $u\in A'\cup B'$ and $v\in \displaystyle \bigcup_{i\geq k+1} L_{i}(F_{h})$. \end{claim} \begin{proof} Let $uv \in E(G[D])$ with $u\in A'\cup B'$ and $v\in \displaystyle  \bigcup_{i\geq k+1} L_{i}(F_{h})$. Suppose, without loss of generality, that $(u,v)\in E(D)$ and $u\in A'$. We have $P_{F_{h}}(v)\cup (u,v) \cup {P_{F_{A}}(u)}_{[z,u]}\cup (z,v_{j})\cup {P_{i}}_{[v_{j},v_{k+r-1}^{i}]}$ contains a $P(k,l,r)$ where $z$ and $v_{j}$ are defined as above, a contradiction. \end{proof}$\blacksquare$\\ Now we are ready to continue the proof of Theorem $3.1$:\\ We have, by definition, $A$ and $B$ has neighbors only in the set $V(\displaystyle \bigcup_{i\geq 1} {P_{i}}_{[v_{1}^{i},v_{k}^{i}]})$. We choose the coloring of $D'$ such that the colors of $V(\displaystyle \bigcup _{i\geq 1} {P_{i}}_{[v_{1}^{i},v_{k}^{i}]})$ are picked from the set of colors $\{1,\dots ,k\}$. Then the colors $\{k+1,\dots ,k+r-1\}$ are not used in coloring of $N_{\{P_{1}\cup \dots \cup P_{h-1}\}}(A\cup B)$.\\ Concerning $D_{h}$, we give $L_{i} (F_{h})$ the color $i$ for every $i$. By claim 3.4, we can give the colors $\{k+1,\dots ,k+r-1\}$ to $A'\cup B'$. We get that $\chi(D'\cup A\cup B)\leq \chi(D'\cup A'\cup B')+\chi(A-A')+\chi(B-B')\leq k+r-1+2+k+l-1=k+k+r+l$.\\ Concerning $C$ and $H$, let $C'=\displaystyle \bigcup _{i\geq l}L_{i}(F_{C})$ and $H'=\displaystyle \bigcup _{i\geq r+1} L_{i}(F_{H})$. Similarly we prove that $uv\notin E(G[D])$ for all $u\in C'$, and $v\in H'$ and that $\chi(C'\cup H')\leq r+1$. Thus, $\chi(C\cup H)\leq l-1+r+r+1=l+r+r$.\\ Therefore, $\chi(D)\leq k+k+r+l+l+r+r=n-1+n-1+r=2n-2+r=2(n-1)+r$.\\ \\ Case 2: $k\leq l\leq r$. Let $A'=\displaystyle \bigcup _{i\geq l+1}L_{i}(F_{A})$, $B'=\displaystyle \bigcup_{i\geq l}L_{i}(F_{B})$, $C'=\displaystyle \bigcup_{i\geq l}L_{i}(F_{C})$, $H'=\displaystyle \bigcup_{i\geq r+1}L_{i}(F_{H})$, and let $F_{h}$ be a final spanning outforest of $D_{h}$. We follow a similar procedure to that followed in the previous case to obtain the following: $\chi(A'\cup B')\leq k+1$, $\chi(A\cup B)\leq k+1+l+l-1=k+l+l,$ and $\chi(D'\cup A\cup B)\leq k+r-1+2+l+l-1=k+r+l+l$.\\ $\chi(C'\cup H')\leq r+1$, and $\chi(C\cup H)\leq l+r+r$. So, $\chi(D)\leq k+r+l+l+r+r+l=n-1+n-1+r-k+1=2n-2+l+r-k=2n-2+l+r-k=2(n-1)+l+r-k$.\\ \\
Case 3: $k\leq r\leq l$. Let $A'=\displaystyle \bigcup _{i\geq l+1}L_{i}(F_{A})$, $B'=\displaystyle \bigcup_{i\geq l}L_{i}(F_{B})$, $C'=\displaystyle \bigcup_{i\geq l}L_{i}(F_{C})$, $H'=\displaystyle \bigcup_{i\geq r+1}L_{i}(F_{H})$, and let $F_{h}$ be a final spanning outforest of $D_{h}$. We follow a similar procedure to the case $1$ and we get the following: $\chi(A'\cup B')\leq k+1$, $\chi(A\cup B)\leq k+1+l+l-1=k+l+l,$ and $\chi(D'\cup A\cup B)\leq k+r-1+2+l+l-1=k+r+l+l$.\\ $\chi(C'\cup H')\leq l+1$, and $\chi(C\cup H)\leq l+l+r$. So, $\chi(D)\leq k+r+l+l +l+l+r=n-1+n-1-k+l+l=2(n+l-1)-k$.
\end{proof} \\ \\ \\

\paragraph{Acknowledgment.} The author would like to acknowledge the National Council for Scientific Research of Lebanon (CNRS-L) and the Agence Universitaire de la Francophonie in cooperation with Lebanese University for granting her a doctoral fellowship.\\ \\

\textbf{Declarations.}\\ \textbf{Conflict of interest} The authors declare that they have no known competing financial interests or personal relationships that could have appeared to influence the work reported in this paper.\\ \textbf{Data sharing} Data sharing not applicable to thus article as no datasets were generated or analysed during the current study.

\end{document}